\documentclass[11pt]{amsart}

\usepackage{amsmath,amsthm,amssymb,mathrsfs,amsfonts}
\usepackage[all,tips]{xy}
\usepackage{graphicx,ifpdf}
\ifpdf
\fi

\newtheorem{theorem}{Theorem}[section]
\newtheorem{lemma}[theorem]{Lemma}
\newtheorem{corollary}[theorem]{Corollary}
\newtheorem{proposition}[theorem]{Proposition}

\theoremstyle{definition}
\newtheorem{definition}[theorem]{Definition}

\theoremstyle{remark}
\newtheorem{remark}[theorem]{Remark}

\numberwithin{equation}{section}

\begin{document}

\title{Equidistribution of expanding measures with local maximal dimension and Diophantine
Approximation}

\author{Ronggang Shi}
\address{Department of Mathematics, Ohio State University, Columbus, Ohio 43210}
 \email{shi@math.osu.edu}


\date{}



\begin{abstract}
\noindent We consider improvements of Dirichlet's Theorem  on space of
matrices $M_{m,n}(\mathbb R) $. It is shown that  for a
certain class of fractals  $K\subset [0,1]^{mn}\subset M_{m,n}(\mathbb
R)$ of local maximal dimension Dirichlet's Theorem cannot be improved 
almost everywhere. This is shown using entropy and dynamics on homogeneous
spaces of Lie groups.

\end{abstract}

\maketitle

\markright{EQUIDISTRIBUTION OF EXPANDING MEASURES}

\section{Introduction}

\subsection{Dirichlet's theorem}
Let $m,n$ be positive integers and denote by $M_{m,n}=\mathbb R^{mn}$ the space $m\times n$ matrices with
real entries. Dirichlet's Theorem (hereafter abbreviated by DT)
on simultaneous diophantine approximations says the following:
\begin{description}
 \item [DT$(m,n)$] Given $Y\in M_{m,n}$ and $N\ge 1$, there exist $\mathbf{q}=
(q_1,\ldots, q_m)\in\mathbb Z^m \backslash \{0\}\subset M_{1,m}$ and $\mathbf{p} =(p_1,\ldots, p_n)\in \mathbb Z^n\subset M_{1,n}$ with
\[
 \|\mathbf q Y+\mathbf p\|\le\frac{1}{N^m}  \quad\mbox{and} \quad \|\mathbf
q\|\le
N^n.
\]
\end{description}
Here and hereafter, unless otherwise specified, $\|\cdot\| $ stands for the
sup norm on $\mathbb R^k$, i.e.\  $\|(x_1,\ldots, x_k)\|=\max_{1\le i\le k}|x_i|.$
We use
$B_s(x)$ (or $B_s$ if $x=0$) to denote the ball of radius $s$ centered at $x$ in this norm.

Given $Y$ as above and a positive number $\sigma<1$, we say DT can be $\sigma$-improved
 for $Y$, and write $Y\in \mathrm{DI}_\sigma(m,n)$ or
$Y\in DI_\sigma$ when the dimensions are clear from the context, if for every
 $N$ large enough one can find  $\mathbf{q}=
(q_1,\ldots, q_m)\in\mathbb Z^m \backslash \{0\}$ and $\mathbf{p} =(p_1,\ldots, p_n)\in
\mathbb Z^n$ with \begin{equation}\label{ineq}
           \|\mathbf q Y+\mathbf p\|\le \frac{\sigma}{N^m} \quad\mbox{and} \quad
\|\mathbf q\|\le \sigma N^n.                                                                         \end{equation}

We say that DT can be improved for $Y$  if
$Y\in DI_\sigma$ for some $0< \sigma <1$. The following theorem of  Davenport
and  Schmidt says that for most
$Y$  DT can not be improved.
\begin{theorem}[\cite{DS}]\label{zero}
 For any $m,n\in \mathbb N$ and positive number $\sigma <1$, the set
$\mathrm{DI}_\sigma (m,n)$ has Lebesgue measure zero.
\end{theorem}
 In fact only the cases with $m=1$ or $n=1$ are proved in \cite{DS}. But the
method there can be generalized to the settings above. After \cite{DS}, there are
different strengthens and generalizations of Theorem \ref{zero}. There are detailed
reviews of the history of these developments in \cite{KW} and \cite{Sh}. In these two
papers, they  successfully strengthen Theorem \ref{zero} for the cases of $m=1$ or
$n=1$. In the case $m=1$, \cite{KW} showed that for a large class of measures  (e.g.\,friendly measures in \cite{KLW}) DT can not be $\sigma$-improved for almost every
element if $\sigma<\sigma_0$ for some positive number $\sigma_0$ depending on the
measure. After that,  Shah improved the result by removing the upper bound $\sigma_0$
for a special kind of measures concentrated on analytic curves.   More precisely,
\begin{theorem}[\cite{Sh}] \label{curve}
 Let $\varphi:[a,b]\to \mathbb R^k$ be an analytic curve such that
$\varphi([a,b])$
is not contained in a proper affine subspace. Then Dirichlet's theorem
DT$(1,k)$ and DT$(k,1)$ can not be improved for $\varphi(s)$ for almost all
$s\in [a,b]$.
\end{theorem}

\subsection{Nonimprovability of DT for fractal measures}
Our aim is to generalize Theorem \ref{zero} in a direction in some sense
 opposite to
Theorem \ref{curve}. Instead of a smooth one-dimensional
submanifold, we are going to consider measures supported on a  full
Hausdorff dimension subset of $M_{m,n}$ and show that for $\mu$ almost every
point DT can not be improved.
Without loss of generality,
we are going to work with measures on $J=[0,1]^{k}\subset \mathbb R^k$.
Let $f$ and $g$ be real valued functions depending on $\epsilon$, then 
$f\ll_\epsilon g$ means $f\le C g$ for some constant $C>0$ depending only on $\epsilon$.

\begin{definition}\label{localIn}
Let $\mu$ be a probability measure on $J$.
 We say $\mu$ has local maximal dimension if there exists $s_0>0$ such that
 for any $\epsilon>0$, $0<\delta<1$, $0<s\le s_0$,
and $x\in J$ one has
\begin{equation}\label{elocmax}
 \mu(B_{\delta s}(x))\ll_\epsilon \delta^{k-\epsilon}
  \mu(B_{ s}(x)).
\end{equation}
We also say $\mu$ has $s_0$-local maximal dimension if $s_0$ is known.
\end{definition}
\begin{remark}
 (\ref{elocmax}) implies   $\mathrm{supp}(\mu)$ has Hausdorff dimension
$k$. 
\end{remark}

In Theorem \ref{conclut} we prove that  DT$(m,n)$ can not be improved almost everywhere if $\mu$  
implies some non-escape of mass property. In particular, we have:
\begin{theorem}\label{obtain}
 Let $\mu$ be a Borel probability measure on $[0,1]^{n}\subset M_{1,n}$ with
local maximal dimension. If $\mu$ is Federer $($see Section \ref{slf}$)$, then DT$(1,n)$ can not be 
improved for $\mu$ almost every element.
\end{theorem}

\subsection{Example of fractal measures}
It is easy to see that the Lebesgue measure on $[0,1]^k$ has local maximal dimension and is Federer. 
Next we give an example (suggested by  Einsiedler) of a fractal measure on $[0,1]$ with the same property but singular to the Lebesgue measure.
First we divide $[0,1]$
 into $3$ subintervals of the same length $\frac{1}{3}$ and cut the middle open
interval out. We denote the remaining two closed subintervals by $[1]$ and
$[2]$ with the natural ordering from left to right. Next we divide these two intervals 
into $5$ subintervals of the same length $\frac{1}{3\cdot 5}$ and
cut the middle interval out.  We denote the
remaining closed intervals inside $[1]$ by $[1,1],[1,2],[1,3],[1,4]$ with the
left to right ordering. We denote the remaining closed intervals inside $[2]$ in a
similar way. In this construction, we allow some overlappings of end points so that all the remaining intervals are closed.

This process is continued for all natural numbers $n$.
That is after $n$-th step we have 
\begin{equation}\label{numnst}
2\cdot 4\cdots(2n)
\end{equation}
   intervals
 which have the same length 
\begin{equation}\label{lennst}
\frac{1}{3}\cdots\frac{1}{2n+1}.
\end{equation}
Each of them is denoted by 
$[y_1,\ldots ,y_{n}]$
 where $1\le y_i\le  2i$. Such a closed interval is said to be of stage $n$.
Then we cut all  of them
 into $2n+3$ subintervals of the
same length and take the middle
open interval out. For the  stage $n$ interval  $[y_1,\ldots
,y_{n}]$, we denote the remaining $2n+2$ subintervals by $[y_1,\ldots
,y_{n}, y_{n+1}]$ with the left to right ordering where $0\le y_{n+1}\le 2n+2$. 
See figure \ref{nstage} for the process of dividing a stage $n$ subinterval.

\begin{figure}
\center{
 \includegraphics{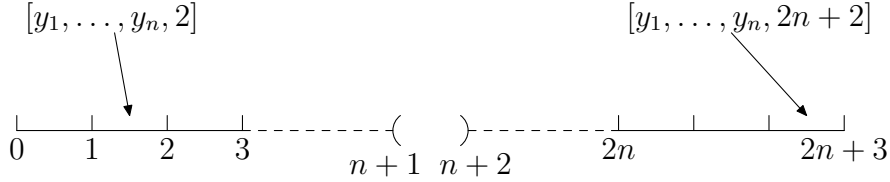}
}
\caption{Stage n interval $[y_1,\ldots,y_n]$ }
\label{nstage}
\end{figure}
We use  $C_n$ to denote the union  of all stage $n$ subintervals. Let
$C=\bigcap_{n}C_n$, then in view of (\ref{numnst}) and (\ref{lennst})
we have
\begin{equation}\label{clmzer}
 m(C)=\lim_{n\to \infty}m(C_n)=\lim_{n\to \infty} 
\frac{2}{3}\cdots\frac{2n}{2n+1}=0
\end{equation}
where $m$ is the Lebesgue measure.
The last equality of (\ref{clmzer}) follows from 
\[
 \left (\frac{2}{3}\cdots\frac{2n}{2n+1}\right)^2\le
\left (\frac{2}{3}\cdots\frac{2n}{2n+1}\right)
\left (\frac{3}{4}\cdots\frac{2n+1}{2n+2}\right)
=\frac{2}{2n+2}\to 0.
\]
One can define a measure $\mu$  on $C\subset [0,1]$ by assigning 
\[
 \mu([y_1,\ldots, y_n])=\frac{1}{2}\cdots\frac{1}{2n}.
\]

\begin{proposition}\label{1.1}
 Let $\mu$ on $ [0,1]$ be the probability measure  above, then
$\mu$ has local maximal dimension  and  is Federer.
\end{proposition}
We omit the proof here, the reader can consult Section 4.1 of the author's thesis \cite{Shi}  for a proof. Many other examples can be constructed in a similar way.
 It is easy to see that local maximal 
dimension is invariant under products. It is mentioned in \cite{KLW} that
Federer is invariant under products, too. So we may see many  examples of
measures on $[0,1]^k$ with local maximal dimension, or in addition Federer and
singular to the Lebesgue measure. 

\subsection{Method of proof}
We are going to translate the diophantine properties  to properties of trajectories for the action of a diagonal matrix on the  homogeneous
space $X=SL(m+n, \mathbb Z)\backslash SL(m+n,\mathbb R)$ in Section
\ref{conclu}. This method is developed in \cite{Da} and \cite{KM} and then was used also  in  \cite{KLW}, \cite{KW} and \cite{Sh} for various kinds of problems.

Our diophantine approximation result follows from an equidistribution result
in Section \ref{maineq}. We put a measure of local maximal dimension on $[0,1]^{mn}$ in the unstable submanifold of $X$. We denote the new measure
by $\nu$ and translate the  property of $\mu$ into the homogeneous setting where we say $\nu$ has  local maximal dimension in the unstable horospherical direction.
We  prove that the average of $\nu$ along the orbit is equidistributed with  respect to the Haar measure $m_X$ if there is no loss of mass.

We will use the entropy theory developed by  Margulis and  Tomanov
 in \cite{MT94}  to prove the equidistribution result. They proved that  the
measure on $X$ of maximal entropy  under diagonal actions is precisely the
Haar measure $m_X$ and the maximal entropy can be computed according to the
 entries of the diagonal matrix. This method will be reviewed  in Section
\ref{cdact}.

To use the entropy theory, we need to show that the average of $\nu$ along the orbit has no loss of mass. In general we do not know whether
this is true  since $X$ is noncompact.   Einsiedler and  Kadyrov are
working on this question under weaker assumptions and have obtained some positive results on special cases.
If $m=1$, we can also use Theorem 3.3 of \cite{KLW} to establish the non-escape of mass property.
 In
Section \ref{slf}, we show that  local maximal dimension and Federer imply absolutely decaying, hence friendly. Therefore with an additional Federer assumption, we get non-escape of mass property and the corresponding diophantine approximation result.

\textbf{Acknowledgements:}
The author would like to thank his advisor Manfred Einsiedler for
his help in preparing this paper and his advice on how
to write articles.

\section{Preliminaries}
We fix a locally compact topological space
$X$ and a continuous map $T:X\to X$. Let $\mathcal B $ stand for the Borel $\sigma$-algebra of $X$. We assume all measures on $X$ are Radon and  
the convergence of  measures  is under the weak$^*$ topology.
\subsection{Equidistribution and non-escape of mass}
A sequence of probability measures $\mu_n$ on $X$  is said to 
be equidistributed with respect to a probability measure $\lambda $, if
\begin{equation}\label{limequ}
 \lim_n\mu_n=\lambda.
\end{equation}
\begin{definition}\label{eqyuav}
Let $\mu$ and $ \lambda$ are probability measures on $X$. We say that $\mu$ is \emph{equidistributed on average}
with respect to $\lambda$ if the sequence
\begin{equation}\label{kavemu}
  \mu_k=\frac{1}{k}\sum_{l=0}^{k-1}T_\ast^l\mu 
\end{equation}
is equidistributed in the sense of $(\ref{limequ})$.
\end{definition}
It is well known that any limit measure  of the sequence $(\ref{kavemu})$ is $T$-invariant.
The following lemma tells us how to compute the value of the limit measure on 
a good Borel set.
\begin{lemma}\label{limit}
 Suppose $\mu_n$ $(n\ge 1)$ and $\mu$ are probability measures on $X$ and $B\in
\mathcal B$ is relatively compact. If $\mu(\partial B) =0$ and $\mu_n\to \mu$,
then
$\mu_n(B)\to \mu(B)$.
\end{lemma}

\begin{definition}\label{nlmass}
For a probability measure $\mu$ on $X$, we say there is no loss of  mass (or non-escape of mass) on
average if for any limit point $\nu$ of the sequence
\[
 \frac{1}{k}\sum_{l=0}^{k-1}T_\ast^l\mu,
 \]
one has $\nu(X)=1$.
\end{definition}

\begin{lemma}\label{noesma}
 Let $\mu_i$ $(i=1,2)$ be probability measures on $X$ and $\mu=c\mu_1+
(1-c)\mu_2$ for some $0<c<1$. If $\mu$ has no loss of mass on average 
then $\mu_i$ $(1\le i\le 2)$ has no loss of mass on average.
\end{lemma}

\subsection{Entropy}
Next we we review the definition of entropy. More details can be found in \cite{EW07} and \cite{Wa}.
Let $\mathcal P\subset \mathcal B$ be a finite or countable partition of $X$ by
Borel
measurable subsets, then the
$entropy$ of $\mathcal P$ is
\[
 H_\mu(\mathcal P)=\sum_{P\in\mathcal P}\mu(P)(-\log \mu(P)).
\]
Let $\mathcal Q$ be another partition. Then the common refinement of $\mathcal P
$ and $\mathcal Q$ is denoted by
\[
 \mathcal P\vee \mathcal Q=\{P\cap Q\ne \emptyset: P\in\mathcal P,Q\in \mathcal
Q\}.
\]
The common refinement of finite  collection of partitions is
defined similarly. 
We use $T^{-1}({\mathcal P})$ to denote the partition of $X$ consisting subsets
of the form
$T^{-1}(P)$ for $P\in\mathcal P$.
\begin{definition}\label{dentro}
 Let $(X,\mathcal B, \mu, T)$ be a measure preserving system and let $\mathcal
P$ be a partition of $X$ with finite entropy, then the entropy of $T$ with
respect to $\mathcal P$ is
\begin{equation}\label{dentro1}
 h_\mu(T,\mathcal P)=\lim_{n\to \infty}
\frac{1}{n}H_\mu \left(\bigvee_{i=0}^{n-1}T^{-i}\mathcal P\right)
\end{equation}
The entropy of $T$ is
\begin{equation}\label{dentro2}
 h_\mu(T)=\sup_{\mathcal P\-:\- H_\mu(\mathcal P)<\infty} h_\mu(T,\mathcal P).
\end{equation}
\end{definition}

\section{Friendly measure and non-escape of mass}\label{slf}
\subsection{Non-escape of mass}
Friendly measure is defined in \cite{KLW}, so let us  review some concepts in that paper.
In this section the norm on $\mathbb R^n$ is  $\|\cdot\|_E$ which is  induced from the standard inner product of $\mathbb R^n$.
For $x\in \mathbb R^n$ and $r>0$, $B(x,r)$ stands for the  open ball of radius $r$ centered at $x$ under $\|\cdot\|_E$. For an affine hyperplane $\mathcal L
\subset \mathbb R^n$, we denote by $d_{\mathcal L}(x)$ the distance from $x$ 
to $\mathcal L$. By $\mathcal L^{(\epsilon)}$ we denote the $\epsilon$-neighborhood of $\mathcal L$, that is the set 
\begin{equation}\label{frie0}
 \mathcal L^{(\epsilon)}
\stackrel{\mathrm{def}}{=}\{x\in \mathbb R^n: d_{\mathcal L}(x)<\epsilon\}.
\end{equation}

Let $\mu$ be a Radon measure on $\mathbb R^n$ and $U$ be an open subset. We say
$\mu$ is \emph{Federer} on $U$ if there exists $c,\beta>0$ such that for all 
$x\in \mathrm{supp}(\mu)\cap U$ and every  $0<\delta \le s$  with $B(x,s)\subset U$ one has 
\begin{equation}\label{frie1}
 \mu(B(x,\delta ))\ge c\left(\frac{\delta}{s}\right)^\beta  \mu(B(x,s)).
\end{equation}
We will say that $\mu$  is \emph{Federer} if for $\mu$-a.e. $x\in X$, there exist
a neighborhood $U$ of $x$ such that $\mu$ is Federer on $U$.

Let $C,\alpha>0$ and $U$ be an open subset of $\mathbb R^n$. We say $\mu$ is 
absolutely $(C,\alpha)$-decaying on $U$ if for any non-empty open ball $B=
B(z,r)
\subset U$ with $z\in \mathrm{supp}(\mu)$, any affine hyperplane $\mathcal L\subset \mathbb R^n$ and any $\epsilon >0$ one has 
\begin{equation}\label{frie2}
 \mu(B\cap \mathcal L ^{(\epsilon)})\le C\left(\frac{\epsilon}{r}\right)^\alpha
\mu(B).
\end{equation}
We will say $\mu$ is absolutely decaying if for $\mu$-a.e. $y_0\in \mathbb R^n$, there exist
a neighborhood $U$ of $y_0$ and $C,\alpha>0$ such that $\mu$ is absolutely $(C,\alpha)$-decaying on $U$.

Friendly measure in \cite{KLW} is defined as Federer, nonplanar and decaying. The
measures interested to us are absolutely decaying which   implies nonplanar and decaying.

The non-escape of mass is related to Theorem 3.3 of \cite{KLW}. The homogeneous space is a special case of Section \ref{conclu}. Here $n>0$,  $G=SL_{n+1}(\mathbb R)$, $\Gamma=SL_{n+1}(\mathbb Z)$ and 
$X=\Gamma\backslash G$. Let $t>0$ and
\begin{equation}\label{frie3}
 a=\mathrm{diag}(e^{t},\cdots, e^{t}, e^{-nt})\in G.
\end{equation}
The dynamical system  is $T=T_a:X\to X$ which sends $x\in X$ to $xa^{-1}$.
We define the following maps from $\mathbb R^n$ to $G$ and $X$:
\[\phi(y)\stackrel{\mathrm{def}}{=}\left(
 \begin{array}{cc}
 I_n & 0 \\
y & 1
\end{array}
\right),\quad 
\tau(y)\stackrel{\mathrm{def}}{=}\Gamma \phi(y).
\]
Recall that $X$ can be identified with the space $\Omega$ of unimodular lattices of $\mathbb R^{n+1}$. For $\epsilon >0$, we define \begin{equation}\label{frie4}
 F_\epsilon\stackrel{\mathrm{def}}{=}\{\Delta \in \Omega: \|v\|_E\ge \epsilon\quad \forall\  v\in \Delta 
\backslash
\{0\}\},
 \end{equation}
i.e., $F_\epsilon$ is the collection of all unimodular lattices in $\mathbb R^{n+1}$ which contain no nonzero vector smaller than $\epsilon$. It is easy to see that 
$\{F_\epsilon\}_{\epsilon>0}$ is an exhaustion of $X$. With these preparations, we can state Theorem 3.3 of \cite{KLW} as follows:

\begin{theorem}\label{tfrie}
 Suppose $\mu$ is a friendly measure on $\mathbb R^n$ and $a$ as in
(\ref{frie3}). Then for $\mu$-almost 
every  $y_0\in \mathbb R^n$, there is a ball $B$ centered at $y_0$ and $\widetilde
C, \alpha>0$ such that for any $l\in \mathbb Z_{\ge 0}$ and $\epsilon>0$,
\begin{equation}\label{frie5}
 \mu(\{y\in B:\tau(y)a^{-l}\not\in F_\epsilon\})\le \widetilde C \epsilon^\alpha.
\end{equation}
\end{theorem}
 Now let us fix a probability measure $\mu$ on $\mathbb R^n$ and assume it is 
friendly. We can cover supp$(\mu)$ by countably many open balls such that 
Theorem \ref{tfrie} holds. Therefore 
given a positive number  $\delta$ (close to $0$), there exist balls $B_1, \ldots, B_m$ such that Theorem \ref{tfrie} holds for all of them with the same $\widetilde
C, \alpha$ and  $\mu(\cup B_i)\ge 1-\delta$. So for any integer $l\ge 0$ and any $\epsilon >0$,
\begin{eqnarray}
  & &\mu(\{y\in \mathbb R^n:\tau(y)a^{-l}\not\in F_\epsilon\})\notag \\
&\le& \delta+\sum_{i=1}^m\mu(\{y\in B_i:\tau(y)a^{-l}\not\in F_\epsilon\}) 
\le \delta+m\widetilde C\epsilon ^{\alpha}\label{frie6}
\end{eqnarray}
This allows us to prove the following non-escape of mass result:
\begin{corollary}\label{cfrie}
 Let $\mu$ be a probability measure on $\mathbb R^n$ and $\tau,a$ as above. If
$\mu$ is friendly, then 
$\nu=\tau_*\mu$ has no loss of mass on average with respect to $T=T_a$.
\end{corollary}
\begin{proof}
 Let $\eta$ be a limit point of the sequence $\frac{1}{k}\sum_{l=0}^{k-1}T_*^l\nu$. 
Without loss of generality we may assume $\eta=\lim_{k\to \infty} 
\frac{1}{k}\sum_{l=0}^{k-1}T_*^l\nu$.

Given $\epsilon >0$, we want to compute $\eta(F_\epsilon)$.  It is easy to see that if $\epsilon_1<\epsilon$, then
$F_\epsilon$ is contained in the interior of $F_{\epsilon_1}$. Therefore we may
assume $\eta(\partial F_\epsilon)=0$. 
$F_\epsilon$ is relatively compact by 
Mahler's criterion (\cite{Ra} Chapter 10).
According to 
Lemma \ref{limit},
\begin{eqnarray}
 \eta(F_\epsilon)&=& \lim_{k\to \infty}\frac{1}{k}\sum_{l=0}^{k-1}T_*^l\nu
(F_\epsilon)  =  \lim_{k\to \infty}\frac{1}{k}\sum_{l=0}^{k-1}T_*^l\tau_*\mu 
(F_\epsilon) \notag\\
&= & \lim_{k\to \infty}\frac{1}{k}\sum_{l=0}^{k-1}
\mu(\{y\in \mathbb R^n: \tau(y)a^{-l}\in F_\epsilon\}) \notag\\
&=& 1-\lim_{k\to \infty}\frac{1}{k}\sum_{l=0}^{k-1}
\mu(\{y\in \mathbb R^n: \tau(y)a^{-l}\not\in F_\epsilon\}).
\label{frie7}
\end{eqnarray}
Apply estimate (\ref{frie6}) for (\ref{frie7}), we have 
\[
 \eta(F_\epsilon)\ge 1-\delta+m\widetilde C \epsilon ^\alpha
\]
for some constants $m,\alpha, \widetilde C>0$ which do not depend on $\epsilon$.
By taking $\epsilon\to 0$ (for those with $\eta(\partial F_\epsilon)=0$), we have
\[
 \eta(X)\ge 1-\delta.
\]
 Since $\delta $ is arbitrary, $\eta (X)=1$.
\end{proof}

\subsection{Local maximal dimension and friendly}
Let $\mu$ be a Radon measure on $[0,1]^n$, then we say $\mu$ is \emph{Federer},
\emph{absolutely decaying} or \emph{friendly} if as a measure on $\mathbb R^n$, it is 
Federer, absolutely decaying or friendly. 
We will show that if $\mu$ has local maximal dimension and  is Federer, then it is
absolutely decaying and therefore friendly. 
To avoid confusion we review some notations. 
We use $\|\cdot\|$ to denote the sup norm on $\mathbb R^n$ and $B_s(x)$ for the ball of radius $s$ center  $x$ under this norm. $\|\cdot\|_E$ stands for the Euclidean norm on $\mathbb R^n$ and $B(x,s)$ stands for the ball under this norm.

If $\mu$
has local maximal dimension, then as a measure on $\mathbb R^n$ it has the following property:
There exists $s_0>0$ such that
 for any $\epsilon>0$, $0<\delta<1$, $0<s\le s_0$,
and $x\in \mathbb R^n$, one has
\begin{equation}\label{lfe1}
 \mu(B_{\delta s}(x))\ll_\epsilon \delta^{n-\epsilon}
  \mu(B_{ s}(x)).
\end{equation}

\noindent
Since $\mu$ is Federer, for $\mu$-a.e. $y\in \mathbb R^n$, there is a neighborhood
$U$ of $y$ such that $\mu$ is Federer on $U$, that is (\ref{frie1}) holds. 

Let us fix $y$ and $U$ as above. Suppose $r_0>0$ such that  $B_{9nr_0}(y)\subset U$ and $9nr_0<s_0$ where  $s_0$ is the upper bound of $s$ in (\ref{lfe1}). Here
the radius
 $9n r_0$ is used so that the balls we are considering below are inside $U$.
In the following three lemmas, 
we  
use (\ref{lfe1}) and (\ref{frie1}) to
show $\mu$ is absolutely $(C,\alpha)$-decaying on $V=B_{r_0}(y)$ for some $C,\alpha>0$.

\begin{lemma}\label{lfl2}
Let $B=B(z,r)\subset V$ where $z\in \mathrm{supp}( \mu)$ and $\mathcal L$ be an
affine 
hyperplane of $\mathbb R^n$.
Suppose $0<\epsilon<r$, then $B\cap \mathcal L^{(\epsilon)} $  can be covered (measure theoretically) by as few as
$2\left(\frac{r}{\epsilon}\right)^{n-1}$ sets of the form $B_{3\epsilon n}(x)$ 
where $x\in  B$.
\end{lemma}
\begin{proof}
 Let us fix some notations first. In a Euclidean space 
with a fixed orthonormal basis ball and box mean the usual figure in Euclidean geometry. We will say $n$-ball or $n$-box if
 we want to emphasize the dimension.  Without loss of generality, we assume
$B\cap \mathcal L^{(\epsilon)} $ is nonempty.

The closure of  $\mathcal L^{(\epsilon)}$ in $\mathbb R^n$
is a family of affine hyperplanes parallel to $\mathcal L$. Each of them is an Euclidean space under the induced inner product if  we fix an origin. We can fix an orthonormal basis for all of them so that we can talk about box and ball 
as above.
Under these frames a hyperplane  intersects $B$ in a ball of radius $\le r$. Let $L$ be a hyperplane such that $L\cap B$ has the  largest area. Since $L\cap B$ is a $(n-1)$-ball of radius $\le r$, it is contained in a $(n-1)$-box of length $2r$. Such  a box can be covered by 
\[
 \le \left(\frac{r}{\epsilon}\right)^{n-1}+1\le 
2\left(\frac{r}{\epsilon}\right)^{n-1}
\]
$(n-1)$-boxes of length $2\epsilon$. From Euclidean geometry, we know each $(n-1)$-box of length $2\epsilon$ is contained in an $(n-1)$-ball of radius
\[
 \epsilon\sqrt{n-1}\le \epsilon n.
\]
 So we can find a covering of $B\cap L$ by $(n-1)$-balls $B_1,\ldots,B_m$  in $L$
centered at $B\cap L$ with radius $\epsilon n$ for some integer $m\le 2\left(\frac{r}{\epsilon}\right)^{n-1}$. 

Assume $B_i$ has center $x_i$, then the ball $B(x_i, 3\epsilon n)$ in $\mathbb R^n$ contains $B_i$.
We claim that $B(x_i, 3\epsilon n)$ for $1 \le i\le m$ cover $B\cap \mathcal L^{(\epsilon)}$. To see this, let $x\in B\cap \mathcal L^{(\epsilon)}$. Since $B\cap L$ has the  largest  area, there exists $b\in B\cap L$ such that 
$\|x-b\|_E< 2\epsilon$. Note  $b\in B_i$ for some $i$, so $\|b-x_i\|_E< \epsilon n$. Therefore
\[
 \|x-x_i\|_E\le\|x-b\|_E+\|b-x_i\|_E< 2\epsilon+\epsilon n\le 3\epsilon n.
\]
The lemma follows from the fact that $B(x_i, 3\epsilon n)\subset B_{3\epsilon n}(x_i)$.
\end{proof}

\begin{lemma}\label{lfl3}
Let $B=B(z,r)\subset V$ where $z\in \mathrm{supp}( \mu)$. If $0< \epsilon <r$ and 
$x\in B$, then
  \begin{equation}\label{lfe7}
 \mu(B_{3\epsilon n}(x))\le C\left (\frac{\epsilon 
}{r}\right)^{n-0.1}\mu(B(z,r))
\end{equation}
where the constant $C$ does not  depend on $B$, $x$ and $\epsilon$. 
\end{lemma}
\begin{proof}
By (\ref{lfe1}),
\begin{equation}\label{lfe3}
 \mu(B_{3\epsilon n}(x))\le C_1 \left (\frac{3\epsilon n}{r}\right)^{n-0.1}
\mu(B_{r}(x))=C\left (\frac{\epsilon }{r}\right)^{n-0.1}\mu(B_{r}(x))
\end{equation}
for some constant $C_1$ and hence $C$ depending on the exponent $0.1$. Since $x\in B=B(z,r)\subset B_r(z)$, we have
$
 B_{r}(x)\subset B_{2r}(z).
$
Apply this for (\ref{lfe3}),
\begin{equation}\label{lfe4}
 \mu(B_{3\epsilon n}(x))\le C\left (\frac{\epsilon 
}{r}\right)^{n-0.1}\mu(B_{2r}(z))\le 
C\left (\frac{\epsilon 
}{r}\right)^{n-0.1}\mu(B(z,2r\sqrt{n}))
\end{equation}
since $n$-box $B_{2r}(z)$ is contained in $n$-ball $B(z,2r\sqrt{n})$.
Recall that $z\in\mathrm{supp}( \mu)$ and $B(z,2r\sqrt{n})\subset U$ 
by the technical choice of $V$. If we take
$2r\sqrt{n}$ and $r$ as radius in (\ref{frie1}), we have
\begin{equation}\label{lfe5}
 \mu(B(z,r))\ge c\left(\frac{r}{2r\sqrt n}\right)^\beta \mu(B(z,2r\sqrt{n}))
\end{equation}
for some $c,\beta>0$ which depend on $U$.
(\ref{lfe5}) implies that 
\begin{equation}\label{lfe6}
 \mu(B(z,2r\sqrt{n}))\le C_2 \mu(B(z,r))
\end{equation}
where $C_2$ depends on $U$.
Combine (\ref{lfe4}) and (\ref{lfe6}), we have
\[
  \mu(B_{3\epsilon n}(x))\le CC_2\left (\frac{\epsilon 
}{r}\right)^{n-0.1}\mu(B(z,r)).
\]
The dependence of $C$ and $C_2$ implies $CC_2$ is independent of $B$, $x$ and 
$\epsilon$.
\end{proof}
\begin{lemma}\label{lfl4}
 $\mu$ is absolutely decaying on $V$.
\end{lemma}
\begin{proof}
Let $B=B(z,r)\subset V$ where $z\in \mathrm{supp}( \mu)$ and $\mathcal L$ be an affine hyperplane of $\mathbb R^n$. Suppose $0<\epsilon<r$, then
by Lemma \ref{lfl2}, we can cover 
$B\cap \mathcal L^{(\epsilon)}$  by balls $B_{3\epsilon n}(x_i)$ for $x_i\in  B$ and $1\le i\le m\le 2
\left(\frac{r}{\epsilon}\right)^{n-1}$.
So 
\begin{equation*}
 \mu(B\cap \mathcal L^{(\epsilon)})\le \sum_{i=1}^m\mu(B_{3\epsilon n}(x_i)).
\end{equation*}
By the estimate for $\mu(B_{3\epsilon n}(x_i))$ in Lemma \ref{lfl3}, we have
 \begin{equation*}
  \mu(B\cap \mathcal L^{(\epsilon)})\le mC\left (\frac{\epsilon 
}{r}\right)^{n-0.1}\mu(B(z,r)).
 \end{equation*}
where $C$ is independent of $B$,  $\mathcal L$ and $\epsilon$.
By the upper bound of $m$ above,
 \begin{equation}\label{lfe8}
  \mu(B\cap \mathcal L^{(\epsilon)})\le 2C\left (\frac{\epsilon 
}{r}\right)^{0.9}\mu(B(z,r)).
 \end{equation}
If $\epsilon \ge r$, (\ref{lfe8})  holds for $C=1$.
\end{proof}
\noindent
Therefore, we have proved that for $\mu$-a.e.\ $y$ there is a neighborhood $V$ of $y$ such that $\mu$ is absolutely decaying on $V$. We summarize the result as 
follows:
\begin{theorem}\label{lft}
Let $\mu$ be a  probability  measure on $[0,1]^n$. If $\mu$ has local maximal 
dimension and is Federer, then $\mu$ is absolutely decaying, hence friendly.
\end{theorem}

\section{Diagonal actions on homogeneous spaces}\label{cdact}

\subsection{General setup for homogeneous spaces}\label{homo}

In this section we  setup the general concepts and notations for 
Lie groups and their homogeneous spaces that are used in Section \ref{maineq}.

Let $G\subset SL(N,\mathbb R)$ be a closed and connected  subgroup with identity element $e$. Let $\Gamma 
\subset G$ be a discrete subgroup and define $X=\Gamma\backslash G$.  Any $g\in G$ acts on $X$ by right translation $g.x=xg^{-1}=\Gamma(hg^{-1})$
for $x=\Gamma h\in X$. 
Recall that $\Gamma$ is a \emph{lattice} if $X$ carries a $G$-invariant probability
measure $m_X$, which is called the \emph{Haar measure} on $X$. From now on we
 assume that the discrete subgroup $\Gamma $ is a lattice. 

We fix a left invariant metric $d^G$ on $G$ and use $B_r^G(x)$ (or $B_r^G$ if $x=e$) to denote the ball of radius $r$ centered at $x\in G$.
We define a metric $d$ on $X$ by
\begin{equation}\label{metrx}
 d(\Gamma g, \Gamma h)=\inf_{\gamma\in \Gamma}d^G(\gamma g, h).
\end{equation}
For any compact subset $K$ of $X$, there exists $r>0$,
such that the map $B_r^G\to X$ defined by sending $g\in G$ to 
$xg$ where $x\in K$ is an isometry. We call $r$ an injectivity 
radius on $K$.

Let $a\in G$ and consider the map $T=T_a:X\to X$ defined by $T(x)=a.x=
xa^{-1}$.  We define the stable horospherical subgroup for $a$ by 
\[
 G^-=\{
g:a^lga^{-l}\to e \mbox{ as } l\to \infty
\}
\]
which is a closed subgroup of $G$. Similarly one can define the unstable horospherical subgroup by
\[
G^+= \{
g:a^lga^{-l}\to e \mbox{ as } l\to -\infty
\}
\]
which is also a closed subgroup of $G$. The centralizer of $a$ is the closed subgroup
\[
 G^0=C_G(a)=\{
h: ah=ha
\}
\]

Next we  define a special kind of diagonalizable elements which  are first defined
by Margulis and Tomanov in \cite{MT94} in the setting of real and p-adic algebraic groups.
Here  we  use the more general concept in \cite{EL08}, Section 7. 
We say that $a$ is $\mathbb R$-semisimple if as an element of $SL(N,\mathbb R)$
$a$ is conjugate to a diagonal element of $SL(N,\mathbb R)$. In particular, this
implies that the adjoint action $\text{Ad}_a$ ($a\mathfrak ga^{-1}$) of $a$ on the Lie algebra $\mathfrak g$ of $G$ has eigenvalues in $\mathbb R$ so is diagonalizable over $\mathbb R$. 
We say furthermore that $a$ is class $\mathscr A$ if the following properties hold:
\begin{itemize}
 \item $a$ is $\mathbb R$-semisimple.
 \item $1$ is the only eigenvalue of absolute value $1$ for Ad$_a$.
 \item No two different eigenvalues of Ad$_a$ have the same absolute value.
\end{itemize}

For a class $\mathscr A$ element $a$ we have a decomposition of the Lie algebra
$\mathfrak g$ into subspaces
\[
 \mathfrak g=\mathfrak g_-\oplus \mathfrak g_0\oplus \mathfrak g_+
\]
where $\mathfrak g_0$ is the eigenspace for eigenvalue $1$, $\mathfrak g_-$ is 
the direct sum of the eigenspaces with eigenvalues less than $1$ in absolute value,
and $\mathfrak g_+$ is the direct sum of the eigenspaces with eigenvalues greater
than $1$ in absolute value. These are precisely the Lie algebras of  
$G^0, G^-, G^+$, respectively.

Here and hereafter, we assume $\mathfrak g_+$ is an eigenspace of $\mathrm{Ad}_a$  and $G^+$ is abelian. Let $t>0$ be the logarithm of the absolute value of the eigenvalue on $\mathfrak g_+$.
We fix a basis $\mathbf{e}_1,\ldots,\mathbf{ e}_n$ of $\mathfrak g_+$
and  use $\|\cdot\|_+$ to denote the sup norm under this basis, i.e.
\begin{equation}\label{normus}
 \|b_1\mathbf{e}_1+\cdots b_n\mathbf{e}_n\|_+=\sup_{1\le i\le n}|b_i|.
\end{equation}
Let $B_s^+$ (or $ B^+(s)$) be the ball of radius $s$ centered at 
zero   of $\mathfrak g_+$
under this norm.
 Similarly we fix a basis consisting of
eigenvetors for $\mathfrak g_0$ and $\mathfrak g_-$. We use $\|\cdot\|_0$ and $\|
\cdot\|_-$ to 
denote the  sup norm under these basis. There are corresponding 
concepts $B^0_s$ and $B^-_s$.

 There exists 
$\alpha>0$ and an open subset $\widetilde G$ of $e$ in $G$  such that the map
\begin{equation}\label{diffeo}
 \varphi:B^-_\alpha+ B_\alpha^0+ B_\alpha^+\to \widetilde G
\end{equation}
which sends $(x,y,z)$ to $\exp x\exp y\exp z$ is a diffeomorphism. 
 $\alpha$ and $\varphi$ are fixed for Section \ref{cdact} and \ref{maineq}. Each 
element of 
$\widetilde G$ naturally corresponds to an element
\[
 x+y+z\in B^-_\alpha+ B_\alpha^0+ B_\alpha^+\subset \mathfrak g
\]
via the above diffeomorphism $\varphi$.

We define the projection map  
$
 \pi:
\widetilde G\to \mathfrak g_+
$
by
\begin{equation}\label{pii}
\pi(\exp u^-\exp u^0 \exp u^+)=u^+
\end{equation}
for 
$u^-\in B_\alpha^-,\  u^0\in 
B_\alpha^0\ \mbox{and}\ u^+\in B_\alpha^+.$
With these definitions we can say that the multiplication in $G$ is local 
Lipschitz in 
the sense of the following lemma:
\begin{lemma}
 \label{regular}
Given $\epsilon , t>0$, there exist $r,s>0$ such that 
\begin{equation}\label{dreg11}
 \exp(B_s^-)\exp(B_s^0)\exp(B^+_{e^ts})B_r^G\subset \widetilde G
\end{equation}
and 
\begin{equation}\label{dreg12}
  \|\pi(h_1)-\pi(h_2)\|_+\le e^\epsilon
\|\pi(h_1g)-\pi(h_2g)\|_+
\end{equation}
for any $h_1,h_2\in 
\exp(B_s^-)\exp(B_s^0)\exp(B^+_{e^ts})$ and $g\in B_r^G$.
\end{lemma}
\noindent
The above lemma follows from the fact that $\pi$ is smooth and we can give 
each space a proper Riemannian metric according to the norm. We omit the proof, and the reader may see Section 5.1 of the author's thesis \cite{Shi} for a detailed proof.

\begin{definition}
 \label{dregular}
 We say that $(r,s)$ is $(t,\epsilon)$-regular, if they satisfy 
(\ref{dreg11}) and (\ref{dreg12}) above.
\end{definition}

In the following lemma we are going to consider more precisely how $\mathrm{Ad}_a$
changes elements of $\mathfrak g_- $ and $\mathfrak g_+$.
\begin{lemma}
 \label{adjoint}
Suppose $u^-\in B_s^-$ and $u^+\in B^+_s$, then 
$\mathrm {Ad}_a(u^-)   \in B_s^-$ and $\mathrm {Ad}_a(u^+)\in 
B^+_{e^ts}$.
\end{lemma}
\begin{proof}
We prove the part concerning $\mathfrak g_+$ and the other part can be proved similarly.
Recall that $t$ is the logarithm of the absolute value of the eigenvalue on $
\mathfrak g_+$. So by  the definition of $\|\cdot\|_+$ in (\ref{normus}),
\[
\| \mathrm {Ad}_a(u^+)\|_+=e^t\|u^+\|_+.
\]
\end{proof}

\subsection{Entropy and measure}\label{enmea}
Let $a\in G$ be a class $ \mathscr A$ element and $T=T_a:X\to X$ be the map
which sends $x\in X$ to $a.x=xa^{-1}$. 
In this section we  review the  results about using entropy to 
classify  $T$-invariant measures on the homogeneous space $X$. 
The method dates back to  Ledrappier and Young \cite{LY85}  who used entropy to classify invariant probability measures on compact Riemannian manifolds under 
a smooth map which answered a question by  Pesin. 
Later their method was adapted by  Margulis and  Tomanov 
in \cite{MT94} to the settings of products of real and $p$-adic 
algebraic groups. In \cite{MT94} measures invariant under 
unipotent flows are classified. Along the way measures of 
maximal entropy for diagonal flows are also characterized. 
A convenient modern reference  of these results is \cite{EL08}.

\begin{theorem}[\cite{MT94}]\label{en-mea}
 Let $\mu$ be a $T$-invariant probability measure on $X$, then 
\begin{equation}\label{entrob}
 h_\mu(T)\le  -\log\big |\mathrm{det}\,\mathrm{Ad}_a|_{\mathfrak g_-} \big |
\end{equation}
and equality holds iff $\mu$ is $G^-$ invariant.
\end{theorem}

A Lie group $G$ which has a lattice as a discrete subgroup is unimodular. This
implies 
$\mathrm{det}(\mathrm{Ad}_g)=1$ for all $g\in G$. Thus
\begin{equation}\label{g-g+}
 -\log\big|\mathrm{det}\,\mathrm{Ad}_a|_{\mathfrak g_-}\big|=
\log\big|\mathrm{det}\,\mathrm{Ad}_a|_{\mathfrak g_+}\big|=nt.
\end{equation}
Since $h_\mu(T)=h_\mu(T^{-1})$, if equality holds in (\ref{entrob}), we will
have a similar equality for $T^{-1}$ which is defined by $a^{-1}$ action. Thus
$\mu$ is invariant under the closed subgroup generated by $G^+$ and $G^-$.
 It is
not hard to see from the definition of $G^+$ and $G^-$ that they are 
$a$-normalized  
subgroups of $G$. Furthermore the closed subgroup generated by them is  normal  since $G$ is connected. 
In the literature,  this subgroup is called the \emph{Auslander normal subgroup}
for the element $a$. In many cases this theorem shows that the Haar measure on $X$ is the unique measure of maximal entropy, e.g.
\begin{corollary}\label{ausl}
Let $\Gamma$ be a lattice of $G$ and $X=\Gamma\backslash G$.
 If the action of the Auslander normal subgroup of $a$ is uniquely ergodic on 
$X$, then
$X$ has a unique measure $m_X$ of maximal entropy under map $T$.
\end{corollary}
\begin{remark}
If the Auslander normal subgroup of $a$ is the whole group $G$, then its action is automatically uniquely ergodic. Hence Corollary \ref{ausl} is true.
\end{remark}

\section{Equidistribution of measures  on homogeneous spaces}
\label{maineq}

In this section notations are the same as in Section \ref{cdact}.
So $G$ is a closed connected linear group with identity $e$, $\Gamma $ is a lattice
of $G$, $X$ is the homogeneous space $\Gamma\backslash G$, and $m_X$ is the probability
Haar measure on $X$. 
Also $a\in G$ is an element of class $
\mathscr A$ and $T=T_a:X\to X$ is the map that sends $x$ to $a.x=xa^{-1}$.
Recall that  we assume $G^+$  is abelian and its Lie algebra $\mathfrak g_+$ is an eigenspace of $\mathrm{Ad}_a $ with dimension $n$.

\subsection{Properties of measures}\label{example}

\begin{definition}\label{wloc-max}
 Suppose $\kappa>0$ and $\mu $ is a Borel probability measure with compact 
support on $X$. We say
 $\mu$ has local dimension $\kappa$ in the unstable horospherical direction if 
there exist   
$s_0>0$ and a finite measure $\lambda$ on $X$ such that  for any  $0<
\tilde s\le s< s_0$, $u\in B_s^+$ with 
$B^+_{\tilde s}+u\subset B^+_s$,  $0<\delta<1$ and $x\in\mathrm{supp}(\mu)$ 
one has \[
         \mu(\exp(B_s^-)\exp(B_s^0)\exp(B^+_{\delta \tilde s}+u).x)
\]
\begin{equation}\label{eloc-max}
          \ll_\kappa \delta^{\kappa}
\lambda(\exp(B_s^-)\exp(B_s^0)\exp(B^+_{\tilde s}+u).x).
        \end{equation}
We will say $\mu$ has local maximal dimension 
in the unstable horospherical direction if there exist $s_0$ and
$\lambda$ as above such that (\ref{eloc-max}) holds
for any $\kappa<n$.
\end{definition}
We say $\mu$ has $s_0$-local  dimension $\kappa$ in the unstable horospherical direction if $s_0$ is known.
 If $s_0<\alpha$ for the   $\alpha$ in  (\ref{diffeo}), then  $\varphi$ can be 
used and  (\ref{eloc-max}) is the same as 
\[
         \mu(\varphi(B_s^-+B_s^0+B^+_{\delta \tilde s}+u).x)
          \ll_\kappa \delta^{\kappa}
\lambda(\varphi(B_s^-+B_s^0+B^+_{\tilde s}+u).x).
        \]

For the fixed basis $\mathbf{e}_1,\ldots,\mathbf{ e}_n$ of $\mathfrak g_+$, we  define a map $\psi:\mathbb R^n\to \mathfrak g_+$ which sends $(x_1,\ldots,x_n)$ to 
$x_1\mathbf{e}_1+\cdots+x_n\mathbf{e}_n$.
 $\psi$ is an isometric isomorphism with respect to the sup norm of $\mathfrak 
g_+$ under the chosen basis.
The composite $\varphi\circ \psi:\mathbb R^n\to G^+\subset G$ is 
a homomorphism of Lie groups as $G^+$ is assumed to be abelian. 
Let us fix some $\mathbf x\in X$ and define $\tau: \mathbb R^n\to X$ that sends
$b\in J$ to $\mathbf x\varphi\circ \psi(b)$. See  Figure \ref{fig6.1} for the relationship of these maps. Recall that
 $B_s(x)$ (or $B_s$ if $x=0$) stands for the ball of 
radius $s$ centered at $x$  in $\mathbb R^n$. 

\begin{figure}
\center
\[
 \xymatrix{
\mathbb R^n \ar[r]^{\tau} \ar[d]^{\psi}  & X\\
\mathfrak g \ar[r]^{\varphi} &G \ar[u]_{\mathbf x\, g}
}
\]
\caption{Relationship of maps}\label{fig6.1}
\end{figure}

\begin{proposition}\label{decompo}
 Let $\mu$ be a Borel probability measure on $J$ with local maximal 
dimension. There exists $s_0>0$ such that  if
$B_\sigma(x)\subset J $ and $\mu(B_\sigma(x))\neq 0$ for some $0<\sigma<s_0$, then $\tau_*\nu$ 
where $\nu=\frac{1}{\mu(B_\sigma(x))}\mu|_{B_\sigma(x)}$
has local maximal dimension in the unstable horospherical direction.
\end{proposition}

\begin{proof}

Suppose $r$ is 
an injectivity radius on $\tau (J)$. We choose some $0<s_0<\alpha/2$ for the $\alpha$ in (\ref{diffeo})  such that
$\mu$ has $s_0$-local maximal 
dimension
and 
\begin{equation}\label{6.1.12}
 \varphi(B_{s_0}^-+ B_{s_0}^0+B_{s_0}^+)\subset B_r^G.
\end{equation}
For $\sigma<s_0$ and $\epsilon>0$, we prove that $\nu$ has $s_0$-local  dimension $n-\epsilon$
in the unstable horospherical direction.
  So it suffices to prove 
\begin{eqnarray}
      & &   \tau_*(\mu|_{B_\sigma(x)})(\varphi(B_s^-+B_s^0+B^+_{\delta \tilde 
s}+u).y)\notag\\
               &\ll_\epsilon& \delta^{n-\epsilon}
\tau_*\mu(\varphi(B_s^-+B_s^0+B^+_{\tilde s}+u).y)\label{prop2}
        \end{eqnarray}
where $ \delta, \tilde s, s,u$ are as in the setting of Definition
\ref{wloc-max} and $y=\tau(b)$ for some $b\in B_\sigma(x)$. 
 
To analyze the $\tau_*(\mu|_{B_\sigma(x)})$ part in (\ref{prop2}), it is convenient to write
\begin{equation}\label{6.1.3a}
 \tau(B_\sigma(x))=\mathbf x\,\varphi\circ \psi(B_\sigma+x)
=y\,\varphi\circ \psi(B_\sigma+x-b).
\end{equation}
Since $b\in B_\sigma+x$, $\sigma<s_0<\alpha/2$ and $\psi $ is an isometry,
\begin{eqnarray}\label{prop3}
 & & \tau(B_\sigma(x))\bigcap \varphi(B_s^-+B_s^0+B^+_{\delta 
\tilde s}+u).y \notag \\
& =&  \varphi(B_\sigma^++\psi(b-x)).y\bigcap \varphi(B_s^-+B_s^0+B^+_{\delta 
\tilde s}+u).y \notag \\
&=& \varphi(B_\sigma^++\psi(b-x)).y\bigcap \varphi(B^+_{\delta 
\tilde s}+u).y\notag \\
&=& y \varphi\circ\psi(B_\sigma +x -b)  \bigcap y \varphi\circ\psi(B_{\delta \tilde s}-c)
\end{eqnarray}
where $c=\psi^{-1}(u)\in \mathbb R^n$. In view of (\ref{prop3}), 
\begin{eqnarray*}
          A_{\delta\tilde s} &
\stackrel{\mathrm{def}}{=} &\tau^{-1}(\varphi(B_s^-+B_s^0+B^+_{\delta 
\tilde 
s}+u).y)\bigcap B_\sigma(x)  
\end{eqnarray*}
consists exactly $z\in B_\sigma(x)$ such that $z-b\in B_{\delta \tilde s}-c $.
So
\begin{eqnarray}\label{6.1.3b}
 A_{\delta\tilde s}= B_{\delta 
\tilde 
s}(b-c)\cap B_\sigma(x)\subset B_{\delta 
\tilde 
s}(b-c).
          \end{eqnarray}
To compute the $\tau_*\mu$ part of (\ref{prop2}), let
\begin{eqnarray}
  A_{\tilde s}&\stackrel{\mathrm{def}}{=}
&\tau^{-1}(\varphi(B_s^-+B_s^0+B^+_{\tilde s}+u).y)
 \supset\tau^{-1}(\varphi\circ\psi (B_{\tilde s}+c).y) \notag \\
&=& \tau^{-1}(\mathbf x
\,\varphi\circ\psi (B_{\tilde s}+b-c))
\supset B_{\tilde s}(b-c).\label{6.1.5}
\end{eqnarray}

Since $\mu$ has local maximal dimension, (\ref{lfe1}) holds, i.e.
\begin{equation}\label{6.1.7}
  \mu(B_{\delta \tilde s}(b-c))\ll_\epsilon \delta^{n-\epsilon}
  \mu(B_{\tilde s}(b-c)).
\end{equation}
By (\ref{6.1.3b}), (\ref{6.1.5}) and (\ref{6.1.7})
\[
 \mu(A_{\delta\tilde s})\le \mu(B_{\delta \tilde s}(b-c))
\ll_\epsilon \delta^{n-\epsilon}
  \mu(B_{\tilde s}(b-c))\le \delta^{n-\epsilon}\mu( A_{\tilde s}).
\]
This completes the proof.
\end{proof}

\subsection{Equidistribution of measures}
\begin{theorem}\label{locdim}
    Let $\mu$ be a Borel probability measure on $X$. Suppose $G^+$
is abelian and $\mathfrak g_+$ is an eigenspace of $\mathrm{Ad}_a$. If $\mu$ has local dimension $\kappa$ in the unstable horospherical direction for some $\kappa>0$ and 
$\rho$ is a limit point of the sequence
$\frac{1}{k}\sum_{l=0}^{k-1}T_\ast^l\mu$ such that $\rho(X)>0$,
then $h_{\nu}(T)\ge \kappa t$ where $\nu=\frac{\rho}{\rho(X)}$.
\end{theorem}
\begin{proof}
Without loss of generality, we may assume 
$$\nu=
\lim_{k\to \infty}\frac{1}{k}\sum_{l=0}^{k-1}T_\ast^l\mu.$$
We fix some $0<\epsilon<\min\{\frac{t}{2},1\}
$ and will construct a
finite partition $\mathcal P$ 
of $X$ such that 
\begin{equation}\label{equi2}
 h_\nu(T,\mathcal P)\ge \kappa t+f(\epsilon) \quad \mbox{with} \quad
\lim_{\epsilon\to 0}f(\epsilon)=0.
\end{equation}
In view of the definition of $h_\nu(T)$ in (\ref{dentro2}) and (\ref{equi2}),
\[
h_\nu(T)\ge \kappa t+f(\epsilon).
\]
Let $\epsilon \to 0$ and we see $h_\nu(T)\ge \kappa t$ which 
completes the proof. 
 The proof of (\ref{equi2})  is divided into
four steps.

\noindent
\textbf{Step one}: Construction of the partition $\mathcal P$. Fix a
compact set $K\supset \mathrm{supp}(\mu)$ with 
$\nu(K)>1-\epsilon^2$. Choose some positive numbers $r$
and $s_0$ such that $2r$
is an injectivity radius on $K$ and $\mu $ has $s_0$-local  dimension $\kappa$
in the unstable horospherical direction.
By shrinking $r$ and $s_0$ 
 we may require that $(r,s_0)$ is $(t,\epsilon)$-regular as in Definition 
\ref{dregular} and $e^ts_0<\alpha$ for the $\alpha$ in (\ref{diffeo}), so that
(\ref{dreg11}), (\ref{dreg12}) hold and $\varphi$ in (\ref{diffeo}) can be used.
Fix some $0<s<s_0$
such that 
\begin{equation}\label{equi3}
\varphi(\widetilde B)^{-1}\varphi( \widetilde B)\subset B_r^G
\end{equation}
 where 
\begin{equation}\label{equi4}
\widetilde B=B_s^-+B_s^0+B^+_{e^ts}
\subset \mathfrak g.
\end{equation}
Consider the covering of $K$ by sets of the
form $\varphi(B).x$ where $x\in K$ and 
\begin{equation}\label{equi5}
B=B_s^-+B_s^0+B^+_{s}\subset \mathfrak g.
\end{equation}
Since $K$ is compact there exists a finite covering
with
centers $x_i$ for $ 1\le i\le q$.
We may assume that $x_1,\ldots,x_p\in \mathrm{supp}(\mu)$ and 
$\varphi(B).x_i$ for $1\le i\le p$ cover $\mathrm{supp}(\mu)$. 
Furthermore by enlarging $s$ a little bit but still
requiring $s< s_0$ and (\ref{equi3}), we may assume that $\nu(\partial (\varphi(B).x_i))=0$ 
for each $i$. Let
\begin{equation}\label{equi6}
\widetilde P_i=\varphi(B).x_i\quad \mbox{for} \quad 1\le i\le q, \quad  \widetilde
P_0=X-K
\end{equation}
 and 
\begin{equation}\label{equi7}
\widetilde
{\mathcal P}=\{\widetilde P_i: 1\le i\le p\}.
\end{equation}
Note that elements of $\widetilde P$ cover supp$(\mu)$.
 The construction of $\mathcal P$
is as
follows: 
\begin{equation}\label{equi8}
P_1=\widetilde P_1, P_2=\widetilde P_2\backslash P_1, P_3=\widetilde P_3\backslash
(P_1\cup P_2),\ldots, P_q=\widetilde P_q\backslash(\bigcup_{i=1}^{q-1}P_i)
\end{equation}
 and $$P_0=
X\backslash\bigcup_{i=1}^qP_i.$$ 
It follows that 
\begin{equation}\label{equi9}
\nu (P_0)\le \nu(X\backslash K)<\epsilon^2.
\end{equation}
Note that $P_0$ may be an empty set if $X$ is compact but we may
assume
$P_i\ne \emptyset $ for each $1\le i\le q$. We set
\[
 \mathcal P=\{P_0,P_1,\ldots, P_q\}.
\]

\noindent
\textbf{Step two}: General estimate.  Let $$\mathcal P_m=\bigvee_{i=0}^{m-1}T^{-i}
\mathcal P.$$ Then from Definition \ref{dentro}
\begin{equation}\label{equi10}
 h_\nu (T,\mathcal P)=\lim_{m\to \infty}\frac{1}{m}H_\nu (\mathcal P_m)=
\lim_{m\to \infty}
\frac{1}{m}\sum_{Q\in \mathcal P_m}\nu(Q)(-\log\nu( Q)).
\end{equation}
In the above equation the sum runs over all the  nonempty sets of the form  
\begin{equation}\label{equi11}
 Q=Q_0\cap T^{-1}Q_1\cap\cdots\cap
T^{-(m-1)}Q_{m-1}
\quad\mbox{where}\quad  Q_i\in \mathcal P. 
\end{equation}
 Let
\begin{equation}\label{equi12}
 \alpha (Q)={\sup_{1\le l\le m}}\frac{|\{0\le i<l: Q_i=P_0\}|}{l},
\end{equation}
\begin{equation}\label{equi13}
 B_\epsilon=\{x\in X: \sup_{l\ge 1}\frac{1}{l}\sum_{i=0}^{l-1}\chi
_{P_0}(T^i x)>\epsilon\}
\end{equation}
where $\chi_{P_0}$ is the characteristic function of $P_0$.
Then $\alpha(Q)>\epsilon$ implies  $Q\subset B_\epsilon$. By the maximal ergodic
theorem, 
\begin{equation}\label{equi14}
 \epsilon\nu(B_\epsilon)\le \nu (P_0)<\epsilon^2,
\end{equation}
which implies $\nu (B_\epsilon)<\epsilon$. Let
\begin{equation}\label{equi15}
 \mathcal Q_m =\{Q\in \mathcal P_m:\alpha(Q)\le \epsilon\},
\end{equation}
then
\[
 \sum_{Q\notin \mathcal Q_m  }\nu (Q)\le \nu(B_\epsilon)<\epsilon.
\]
Therefore
\begin{equation}\label{equi16}
 \sum_{Q\in\mathcal Q_m }\nu (Q)> 1-\epsilon
\end{equation}
In view of (\ref{equi10}) and (\ref{equi16}), an estimate of $\nu(Q)$ for  $Q\in \mathcal Q_m $ will be enough to prove (\ref{equi2}) and hence the theorem.

\noindent
\textbf{Step three}: Estimate of $\nu (Q)$  for $Q\in \mathcal Q_m $ where $Q$ is in the form
of (\ref{equi11}).
Recall that \[
             Q=Q_0\cap \cdots\cap Q_{m-1}=P_{j_1}\cap \cdots\cap P_{j_{m-1}}
\subset \widetilde P_{j_1}\cap \cdots\cap \widetilde P_{j_{m-1}}
            \]
where $\widetilde P_{j_i}$ is the open subset defined in (\ref{equi6}).
For simplicity of notations we set $\widetilde Q_i=\widetilde P_{j_i}$, so
$\widetilde Q_i=\varphi(B).y_i$ for some $y_i\in K$ if $\widetilde Q_i
\ne \widetilde P_0$. Under these notations
 \[
                                                        Q\subset
\widetilde Q_0\cap\cdots\cap T^{-(m-1)}\widetilde Q_{m-1}
\stackrel{\mathrm{def}}{=}\widetilde Q.
                                                       \]
Since  $\widetilde Q_i$ is open and $\widetilde Q_0\ne \widetilde P_0$,
we  may assume $\widetilde Q=\varphi(U).y_0$ for some open subset $U\subset B$.

Let $N=N(Q)\stackrel{\mathrm{def}}{=}|\{i\in \mathbb Z:0\le i< m,Q_i=P_0\}|$. Since $Q\in \mathcal Q_m $, we have
\begin{equation}\label{equi17}
 N \le m\epsilon.
\end{equation}
Then by Lemma \ref{prep2}, $U$ can be covered by as few as 
\begin{equation}\label{equi18}
2^{Nn}e^{ntN+
\epsilon (m-1-N)n} <2^{Nn} e^{ntN+ mn\epsilon}
\end{equation}
  tube-like sets (see (\ref{condition}) for the precise definition) of 
the form 
\begin{equation}\label{equi181}
 B_s^-+B_s^0+B^+(e^{-(t-\epsilon)(m-1)}s)+u\subset B
\end{equation}
where $u\in B^+_s$.
 Let us fix such a covering $\mathcal
R$ of $U$, then (\ref{equi17}) and (\ref{equi18}) imply
\begin{equation}\label{index}
|\mathcal R| \le 2^{Nn}e^{ntN+ mn\epsilon} \le e^{ntm\epsilon+
m n\epsilon +m  n\epsilon\log 2}=e^{m\epsilon A}
\end{equation}
where $A=nt+n+n\log 2$ is a constant since the system $T:X\to X$ is fixed.
Recall that  $\nu (\partial \widetilde Q)=0$, by Lemma \ref{limit} we have
\begin{eqnarray}
\nu(Q)\le \nu(\widetilde Q)\ll_{\rho(X)}\lim_{k\to \infty}\frac{1}{k}\sum_{l=0}^{k-1}T_\ast^l\mu(\widetilde Q)
=\lim_{k\to \infty}\frac{1}{k}\sum_{l=0}^{k-1}T_\ast^l\mu(\varphi(U).y_0)
 \notag\\
\le\limsup_{k\to
\infty}\frac{1}{k}\sum_{l=0}^{k-1}\sum_{R\in \mathcal R}T_\ast^l
\mu(\varphi(R).y_0) 
=\limsup_{k\to
\infty}\frac{1}{k}\sum_{l=0}^{k-1}\sum_{R\in \mathcal R}
\mu(T^{-l}(\varphi(R).y_0))\notag
\end{eqnarray}

Recall that the elements of $\widetilde P$ defined in (\ref{equi7}) cover the support
 of $\mu$, so for each $R\in \mathcal R$ we have
\[
 \mu(T^{-l}(\varphi(R).y_0))=
\sum_{P\in \widetilde{ \mathcal P}}\mu(T^{-l}(\varphi(R).y_0)\cap P).
\]
Therefore 
\begin{equation}\label{equi19}
 \nu(Q)\ll_{\rho(X)}\limsup_{k\to
\infty}\frac{1}{k}\sum_{l=0}^{k-1}\sum_{R\in \mathcal R}
\sum_{P\in \widetilde{ \mathcal P}}\mu(T^{-l}(\varphi(R).y_0)\cap P).
\end{equation}
Since $|\widetilde P|$ is fixed and  $|\mathcal R|$ is bounded above efficiently in (\ref{index}),  it suffices to estimate
\begin{equation}\label{equi21}
\mu(T^{-l}(\varphi(R).y_0)\cap P)
\end{equation}
 for each $R$ and $P$.
So let us fix some $R\in \mathcal R$ in the form of (\ref{equi181}) and 
$P=\varphi(B).y\in \widetilde P$ for some $y\in \mathrm{supp} (\mu)$. We first  cover $B$ by tube-like sets
of the form 
\begin{equation}\label{equi22}
  V=B_s^-+B_s^0+B^+(e^{-tl}s )+v\subset B \quad \mbox {where}\quad v\in  B^+_s
\end{equation}
in a way that there are not many (bounded absolutely) overlaps.
For an interval $B_s^{\mathbb R}\subset \mathbb R$, 
 it can be covered by  $\le e^{tl}+1$ intervals of the form
\[
 B_{e^{-tl}s}^{\mathbb R}+u\subset 
B_s^{\mathbb R} \quad\mbox{where}\quad u\in B_s^{\mathbb R}
\]
such that each of them intersects at most $2$ of others. 
Since $B^+_s$ is the same as a direct product of $n$ copies of $B_s^{\mathbb R}$,
 we see that $B^+_s$ can be covered by as few as
\begin{equation}\label{equi23}
  (e^{tl}+1)^n
\end{equation}
sets of the form (\ref{equi22}) and each of them intersects at most $3^n$ elements
including itself. This not many overlapping property later will give
(\ref{equi27}).
Now let us fix a covering $\mathcal E$ of $B$ as above, then
\begin{equation}\label{equi24}
 T^{-l}(\varphi(R).y_0)\cap P \subset \bigcup _{E\in \mathcal E}
 T^{-l}(\varphi(R).y_0)\cap \varphi(E).y.
\end{equation}
Let us fix some $E\in \mathcal E$ in the form of (\ref{equi22}). Then by Lemma
\ref{prep3} and Remark \ref{rprep3},
\begin{equation}\label{equi25}
 T^{-l}(\varphi(R).y_0)\cap \varphi(E).y\subset \varphi(W).y
\end{equation}
where $W$ is a tube-like set  of the form
\begin{equation}\label{equiw}
  B_s^-+B_s^0+B^+(e^{-t(m-1+l)+m\epsilon}s )+w\subset E\subset B
\end{equation}
with $w\in B^+_s\cap E$.

By assumption $\mu$ has local  dimension $\kappa$ in the unstable horospherical dimension, so
according to Definition \ref{wloc-max} there exists a finite measure
$\lambda$ on $X$ such that  
\[
\mu(\varphi(W).y)
=\mu( \varphi( B_s^-+B_s^0+B^+(e^{-t(m-1+l)+m\epsilon}s)+w).y)
\]
\begin{equation}\label{equi26}
 \ll_\kappa (e^{-t(m-1)+m\epsilon })^{\kappa}
\lambda(\varphi(B_s^-+B_s^0+B^+(e^{-tl}s )+w).y).
\end{equation}

Strictly speaking $w$ in (\ref{equiw})
 depends on $R$, $E$ and $P$, but we will index it by $E$ 
for simplicity since we are trying to estimate (\ref{equi21}) where $R$ and $P$ are fixed.
As the multiplicity of the intersections of the sets in $\mathcal E$ are bounded
by $3^n$, we have
\begin{equation}\label{equi27}
\sum_{E\in \mathcal E}\lambda (\varphi(
B_s^-+B_s^0+B^+(e^{-tl}s)+w_E).y)\ll 1.
\end{equation}
Now combining (\ref{equi24}), (\ref{equi25}), (\ref{equi26}) and
(\ref{equi27}) we have
\begin{eqnarray}\label{equi28}
\mu(T^{-l}(\varphi(R).y_0)\cap P)
&\le & \sum_{E\in \mathcal E}\mu(\varphi(R\cap E).y_0\cap \varphi(E).y)\notag\\
& \ll_{\kappa} & e^{(-t(m-1)+m\epsilon )\kappa}.
\end{eqnarray}
 By (\ref{index}), (\ref{equi19}) and (\ref{equi28}),
\begin{eqnarray}\label{locdim1}
 \nu(Q)&\ll_{\kappa,\rho(X)}& |\widetilde {\mathcal  P}||\mathcal R|
\exp\big({(-t(m-1)+m\epsilon )\kappa} \big)\notag \\
&=&  |\widetilde {\mathcal  P}|
\exp(-m\kappa t+m\epsilon \kappa +m\epsilon A+\kappa t).
\end{eqnarray}
Since $t$ and $\kappa$ are fixed, $e^{\kappa t}$ is a constant. We are 
trying to estimate $h_\nu(T,\mathcal P)$, so the number $|\widetilde P|$ 
determined by $\mathcal P$ is fixed.
Therefore,
\begin{equation}\label{equi29}
 \nu(Q)\ll_{\kappa,\rho(X), \mathcal P}
\exp\big(m(-\kappa t+\epsilon \kappa+\epsilon A)\big)
\end{equation}
for any $Q\in \mathcal Q_m $.

\noindent
\textbf{Step four}: Conclusion.
With the results of step three we can complete the estimate of $h_\nu(T,\mathcal P)$. By (\ref{equi10}) and (\ref{equi15})
\begin{equation}\label{equi30}
h_\nu(T,\mathcal P)\ge \liminf_{m\to \infty}
 \frac{1}{m}\sum_{Q\in \mathcal Q_m }\nu(Q)(-\log\nu( Q)).
\end{equation}
By (\ref{equi29}), 
\begin{equation}\label{5.394}
-\log \nu(Q)\ge
m(\kappa t-\epsilon \kappa-\epsilon A) +M
\end{equation}
for some constant $M$ depending on $\kappa$, $\rho(X)$ and $\mathcal 
P$.
In view of (\ref{equi30}) and (\ref{5.394}),
\begin{eqnarray*}
 h_\nu(T,\mathcal P)&\ge&
(\kappa t-\epsilon \kappa-\epsilon A)\liminf_{m\to\infty}\sum_{Q\in \mathcal Q_m }\nu(Q) \\
&\ge&
(\kappa t-\epsilon \kappa-\epsilon A)(1-\epsilon)
\end{eqnarray*}
where the last inequality follows from (\ref{equi16}).
Note that \[
       \lim_{\epsilon\to 0}(-\epsilon \kappa-\epsilon A)(1-\epsilon)
-\kappa t\epsilon=0.
      \]
This establishes  (\ref{equi2}) hence the theorem.
\end{proof}
\begin{theorem}\label{wequi}
   Let $\mu$ be a Borel probability measure on $X$. Suppose $G^+$
is abelian, $\mathfrak g_+$ is an eigenspace of $\mathrm{Ad}_a$, and  the action of the Auslander normal subgroup is uniquely 
ergodic. If  $\mu$ has
local maximal dimension in the unstable horospherical direction
and there is no loss of mass  on average with respect to $T$, then
\[
 \lim_{k\to \infty}\frac{1}{k}\sum_{l=0}^{k-1}T_\ast^l\mu =m_X.
\]
\end{theorem}
\begin{proof}
   Let $\nu $ be a limit  point of  the sequence
$
\frac{1}{k}\sum_{l=0}^{k-1}T_\ast^l\mu
$ under the weak$^\ast$ topology. The assumption about no loss of mass
implies that $\nu$ is a  probability measure on $X$. 
According to Definition \ref{wloc-max} and the assumption about the measure,
$\mu$ has local dimension $\kappa$ in the unstable horospherical direction for any $\kappa<n$.
Therefore 
 Theorem \ref{locdim} implies  $h_\nu(T)\ge \kappa t$ for any $\kappa<n$.
So 
$
 h_\nu(T)\ge nt.
$
On the other hand,
Theorem \ref{en-mea} and (\ref{g-g+}) imply
$
 h_{\nu}(T)\le nt.
$
Thus $h_{\nu}(T)= nt$.
By assumption, the action of the Auslander normal subgroup is uniquely ergodic, so $\nu=m_X$ by Corollary \ref{ausl}.  
\end{proof}

\begin{remark} The no loss of mass assumption is superfluous
in many cases, see Corollary \ref{cfrie}. 
\end{remark}

\subsection{Proof of lemmas}

In this section we are going to prove the lemmas that are  used in the proof of Theorem 
\ref{locdim}.
Before doing this, let us fix some notations according to 
the construction of $\mathcal P$ in step one of the proof. Suppose $0<
\epsilon < \min\{\frac{t}{2},1\}$,
 $K$ is a compact subset 
of $X$ and $2r$ be an injectivity radius on $K$. 
Let $0<s<e^{-t}\alpha$
for the $\alpha $ in (\ref{diffeo}) and set
  $
\widetilde B= B_s^-+B_s^0+B^+_{e^t s}$, $B=
 B_s^-+B_s^0+B^+_{ s}$ so that $\varphi$ can be used for elements of $\widetilde 
B$.  
An element of $B$ is usually represented by 
$u^-+u^0+u^+$ where $u^-\in \mathfrak g_- , u^0\in \mathfrak g_0$, and
$u^+\in \mathfrak g_+$. This will
be referred to as the standard representation of  elements in
$\mathfrak g$.  
We also assume
\begin{equation}\label{tilde}
\varphi(\widetilde B)^{-1}\varphi(\widetilde B)
\subset B_r^G
\end{equation}
and 
$(r,s)$ is
$(t,\epsilon)$-regular so that (\ref{dreg11}) and (\ref{dreg12}) hold.
 An open
subset of $B$ is called \emph{tube-like} if it is of the form 
$
 B_s^-+B_s^0+B^+_{\tilde s}+u
$
where  
\begin{equation}
 \label{condition}
u\in \mathfrak g_+ \quad \mbox{and}\quad
B^+_{\tilde s}+u\subset B_s^+.
\end{equation}
For $g,h\in G$, we use $\eta_g(h)$ to denote $ghg^{-1}$. In the proof of the  following lemmas the assumption $G^+$ is abelian is used.
\begin{lemma}\label{prep3}
Suppose $m\ge 0$, $l\ge 1$ and we have  tube-like sets
\begin{equation}\label{prep3.1}
 V=B_s^-+B_s^0+B^+(e^{-(t-\epsilon)m}s)+v
\end{equation}
\begin{equation}\label{prep3.2}
 W=B_s^-+B_s^0+B^+(e^{-t(l-1)}s)+w.
\end{equation}
Then for any $x,y\in K$, 
\begin{equation}\label{prep3.a1}
T^{-l}(\varphi(V).x)\cap \varphi(W).y=\varphi(U).y
\end{equation}
 where $U$ $($possibly empty$)$ is contained in a tube-like set of the form  
\begin{equation}\label{prep3.3}
 B_s^-+B_s^0+B^+(e^{-t(m+l)+(m+1)\epsilon}s)+u\subset W.
\end{equation}
\end{lemma}
\begin{remark}\label{rprep3}
  It is obvious that the conclusion is still valid if we replace $W$ by 
\[
 W_1=B_s^-+B_s^0+B^+(e^{-tl}s )+w\subset W.
\]
\end{remark}

\begin{proof}
We may assume $T^{-l}(\varphi(V).x)\cap \varphi(W).y\ne \emptyset$, otherwise the conclusion is trivial. 
Let $g\in \varphi (U)\subset \varphi(W)$. Suppose $c=\exp (w)$, then from the shape of $W$ in 
(\ref{prep3.2}) we have 
\begin{equation}\label{prep3.4}
 \pi(gc^{-1})\in B^+({e^{-t (l-1) }s}).
\end{equation}
So $gc^{-1}=\varphi(w^-+w^0+w^+)$ where $w^-+w^0+w^+$ is the standard representation for $\varphi^{-1}(gc^{-1})$ and $w^+\in B^+(e^{-t(l-1)}s)
.$ Thus \begin{eqnarray}
         a^lgc^{-1}a^{-l} & = &\eta_{a^l}(\varphi(w^-+w^0+w^+)) \notag\\
&=&\eta_{a^l}(\exp w^-)\eta_{a^l}(\exp w^0)\eta_{a^l}(\exp w^-)
\notag \\
& = & \varphi(\mathrm{Ad}_{a^l}(w^-)+w^0+\mathrm{Ad}_{a^l}(w^+))
\label{prep3.5}
        \end{eqnarray}
Therefore,
\begin{equation}\label{5.3t.0}
 a^lg.y=(a^lgc^{-1}a^{-l})a^lc.y=
 \varphi(\mathrm{Ad}_{a^l}(w^-)+w^0+\mathrm{Ad}_{a^l}(w^+))a^lc.y
.
\end{equation}
According to Lemma \ref{adjoint}, $\mathrm{Ad}_{a^l}(w^-)\in B^-_{s}$ and 
$\mathrm{Ad}_{a^l}(w^+)\in B_{e^ts}^+$. So 
\begin{equation}\label{prep3.6}
 \mathrm{Ad}_{a^l}(w^-)+w^0+\mathrm{Ad}_{a^l}(w^+)\in \widetilde B.
\end{equation}
Since $g\in \varphi(U)$, 
 (\ref{prep3.a1}) implies $g.y\in T^{-l}(\varphi(V).x)$. So
\begin{equation}\label{prep3.7}
 a^lg.y=\varphi( \tilde v ).x \quad  \mbox{for some }  \tilde v \in V\subset 
B\subset\widetilde B.
\end{equation}
From the two expressions of $a^lg.y$ in (\ref{5.3t.0}) and (\ref{prep3.7}), we have
 \begin{equation}\label{prep3.8}
    a^lc.y= \varphi(\mathrm{Ad}_{a^l}(w^-)+w^0+\mathrm{Ad}_{a^l}(w^+))^{-1}
\varphi( \tilde v ).x.
   \end{equation}
By (\ref{tilde}), (\ref{prep3.6})  and (\ref{prep3.7}), 
\begin{equation}\label{prep3.89}
 \varphi(\mathrm{Ad}_{a^l}(w^-)+w^0+\mathrm{Ad}_{a^l}(w^+))^{-1}
\varphi( \tilde v )\in B^G_r.
\end{equation}
Therefore,
\begin{equation*}
 a^lc.y=k.x \quad \mbox{for some }k\in B^G_r
\end{equation*}
and
\begin{equation}\label{5.2t.1}
a^lg.y=a^lgc^{-1}a^{-l}(a^lc.x)=a^lgc^{-1}a^{-l}k.x.
\end{equation}
Let $h\in\varphi(U)$ be another element, then 
\begin{equation}\label{5.2t.2}
 a^lh.y=a^lhc^{-1}a^{-l}(a^lc.x)=a^lhc^{-1}a^{-l}k.x.
\end{equation}
By (\ref{tilde}), (\ref{prep3.5}), (\ref{prep3.6}) and similar results for $h$,
\begin{equation}\label{5.2t.h}
 a^lgc^{-1}a^{-l}, a^lhc^{-1}a^{-l}\in B^G_r.
\end{equation}
Since $(r,s)$ is $(t,\epsilon)$-regular (see Definition \ref{dregular}), we have
\begin{equation}\label{prep3.9}
 \|\pi(a^lgc^{-1}a^{-l})-\pi(a^lhc^{-1}a^{-l})\|_+ \le
e^\epsilon  \|\pi(a^lgc^{-1}a^{-l}k)-\pi(a^lhc^{-1}a^{-l}k)\|_+.
\end{equation}
In view of  (\ref{5.2t.h}) and the fact $k\in B_r^G$,
 $$a^lgc^{-1}a^{-l}k, a^lhc^{-1}a^{-l}k\in B^G_{2r}.$$
Recall that $
                         a^lg.y, a^lh.y\in \varphi(V).x
                        $
and $2r$ is an injectivity radius of $x$, so \ref{5.2t.1}
and \ref{5.2t.2} imply   $a^lgc^{-1}a^{-l}k, a^lhc^{-1}a^{-l}k\in \varphi(V)$.
According to the shape of $V$ in (\ref{prep3.1}), we have
\begin{equation}\label{5.2t.3}
 \|\pi(a^lgc^{-1}a^{-l}k)-\pi(a^lhc^{-1}a^{-l}k)\|_+ 
\le 2 e^{-(t-\epsilon)m}s.
\end{equation}
By (\ref{prep3.9}), (\ref{5.2t.3}) and Lemma \ref{adjoint},
\begin{eqnarray*}
 \|\pi(g)-\pi(h)\|_+ & = & \|\pi(gc^{-1})-\pi(hc^{-1})\|_+ \\
& \le & e^{-tl}\|\pi(a^lgc^{-1}a^{-l})-\pi(a^lhc^{-1}a^{-l})\|_+ \\
& \le & e^{-tl} e^\epsilon\|\pi(a^lgc^{-1}a^{-l}k)-\pi(a^lhc^{-1}a^{-l}k)\|_+\\
& \le & e^{-tl} e^\epsilon 2 e^{-(t-\epsilon)m}s \\
& = &  2 e^{-t(m+l)+(m+1)\epsilon}s.
\end{eqnarray*}
\end{proof}

\begin{lemma}\label{prep1}
 Suppose $Q_i=\varphi(B). x_i$ for $x_i\in K$ and $0\le i\le m$, which are open
subsets of $X$. If $Q=Q_0\cap
T^{-1}Q_1\cap\cdots\cap
T^{-m}Q_{m}$, then $Q\subset \varphi(U).x_0$ for some tube-like set $U
\subset B$ 
of the form 
$$ B_s^-+B_s^0+B^+(e^{-(t-\epsilon)m}s)+u.$$
\end{lemma}
\begin{proof}
 The lemma is proved by induction on $m$. If $m=0$, then $Q_0=
\varphi(B).x_0$ and the lemma is true in this case.

Now assume the lemma is true for $m-1$, then we may assume
\[
 Q_1\cap T^{-1}Q_2\cap\cdots\cap T^{-(m-1)}Q_m\subset \varphi(V).x
\]
where 
\begin{equation}\label{prep1.a}
 V=B_s^-+B_s^0+B^+(e^{-(t-\epsilon)(m-1)}s)+v
\end{equation}
is a tube-like set of $B$. 
It follows from Lemma \ref{prep3} ($m$ and $l$  there equal $m-1$ and $1$)  that 
$$Q\subset Q_0\cap T^{-1}(\varphi(V).x)\subset \varphi(U).x_0
$$
for some tube-like set 
\[
 U=B_s^-+B_s^0+B^+(e^{-tm+m\epsilon}s )+u.
\]

\end{proof}

\begin{lemma}\label{prep2}
 Let $\mathcal N$ be a subset of $\{1,\ldots, m\}$ with $N$ elements, $Q_i$ be an 
open subset of $X$ for $0\le i\le m$ such that $Q_i=\varphi(B).x_i$ for some $x_i\in K$
if $i\not \in \mathcal N$. Let
\[
 Q=Q_0\cap
T^{-1}Q_1\cap\cdots\cap
T^{-m}Q_{m}=\varphi(U).x_0
\]
for some open subset $U$ of $B$. Then $U$ can be covered by as few as 
\begin{equation}\label{prep2.0}
 2^{Nn}e^{ntN+\epsilon (m-N)n}
\end{equation}
tube-like sets of the form
\[
 B_s^-+B_s^0+B^+(e^{-(t-\epsilon)m}s)+u
\]
where $u\in B^+_s$.
\end{lemma}
\begin{proof}
Let $D$ be the total number of blocks of numbers in $\mathcal N$, i.e.\  $D=1$ if 
$\mathcal N=\{i,i+1,\dots, i+j\}\subset \mathbb N$.
We are going to prove the lemma by induction on $D$. If $D=0$, then $N=0$ and it is
proved in Lemma \ref{prep1}.

Now suppose $D>0$ and the lemma is true for $D-1$. Let $i+1, \ldots, i+j$ be the first block such that $\{i+1, \ldots, i+j\}\subset \mathcal N$. So 
\begin{equation}\label{prep2.1}
k\not \in 
\mathcal N\mbox{ if }k\le i\quad\mbox{and}\quad i+j+1\not \in \mathcal N.
\end{equation}

\noindent Lemma \ref{prep1} applies to $Q_0\cap
T^{-1}Q_1\cap\cdots\cap
T^{-i}Q_{i} = \varphi(W).x_0$ and tells us that $W$ is contained  a tube-like
set of the form
\begin{eqnarray*}
\empty & B_s^-+B_s^0+B^+(e^{-(t-\epsilon)i}s)+u.
\end{eqnarray*}
For the interval $B^{\mathbb R}_{e^{-(t-\epsilon)i}s}$,  it can be covered by as few as 
\[
  2\frac{e^{-(t-\epsilon)i}}{e^{-t(i+j)}}=2e^{tj+\epsilon i}
\]
open intervals of the form $B^{\mathbb R}_{e^{-t(i+j)}s}+b\subset 
B^{\mathbb R}_{e^{-(t-\epsilon)i}s}$ where $b\in \mathbb R$.
Since $B^+(e^{-(t-\epsilon)i}s) $  is isomorphic to the product of $n$ copies of 
$B_{e^{-(t-\epsilon)i}s}$, it can be covered by 
 as few as
\begin{equation}
 \label{prep2.2}
(2e^{tj+\epsilon i})^n=2^ne^{ntj+\epsilon in}
\end{equation}
tube-like sets of the form 
\begin{equation}
 \label{prep2.3}
B_s^-+B_s^0+B^+(e^{-t(i+j)}s )+w.
\end{equation}
Let $W_1$ be one of them. 

Now let us consider 
\[
 Q_{i+j+1}\cap
T^{-1}Q_{i+j+2}\cap\cdots\cap
T^{-(m-i-j-1)}Q_{m}= \varphi(V).x_{i+j+1}
\]
for some open subset $V$ of $B$.  The induction hypothesis implies that $V$ can
be covered by as few as 
\begin{equation}
 \label{prep2.4}
2^{(N-j)n}e^{nt(N-j)+\epsilon (m-i-j-1-(N-j))n}=2^{(N-j)n}e^{nt(N-j)+\epsilon (m-i-N-1)n}
\end{equation}
tube-like sets of the form 
\begin{equation}
 \label{prep2.5}
B_s^-+B_s^0+B^+(e^{-(t-\epsilon)(m-i-j-1)}s)+v.
\end{equation}
Let $V_1$ be one of them.
As the product of (\ref{prep2.2}) and (\ref{prep2.4}) is bounded by the number in
(\ref{prep2.0}), it remains to see
\[
 \varphi(W_1).x_0\cap T^{-i-j-1}(\varphi(V_1).x_{i+j+1})=\varphi(U).x_0
\]
for some open subset $U$ which is  contained in a tube-like set of the
 form
\begin{equation}\label{prep2.f}
  B_s^-+B_s^0+B^+(e^{-(t-\epsilon)m}s)+u.
\end{equation}
Lemma \ref{prep3} ($m$ and $l$ there equal $m-i-j-1$ and $i+j+1$) implies that $U$ is a contained in a tube-like set 
\[
 B_s^-+B_s^0+B^+(e^{-(t-\epsilon)m+(m-i-j)\epsilon}s)+u
\]
which is a subset of (\ref{prep2.f}).
\end{proof}

\section{Applications}\label{conclu}

In this section we are going to interpret the improvements of DT in the setting of  homogeneous space. Let $G=SL(m+n,\mathbb R)$, $\Gamma=SL(m+n,\mathbb Z), X=\Gamma\backslash
G$ and $m_X$ be the usual probability Haar measure on $X$. 
$G$ acts on $\mathbb R^{m+n}$ (considered as $M_{1,m+n}$) by 
$g(\xi)=\xi g$ as matrix multiplication.
Let $\Omega$ be the set of unimodular lattices in $\mathbb R^{m+n}$.
$G$ acts on 
$\Omega$ by $g(\Delta)=\Delta g=\{vg:v\in \Delta\}$.
$G$ acts transitively on $\Omega$   and the stabilizer of $\mathbb Z^{m+n}$ is 
$\Gamma$.
Thus $\Gamma\backslash G\cong \Omega$ as a set. We endow $\Omega$ with the natural locally compact topology of $\Gamma\backslash G$. In
 this topology, a sequence $\{\Delta_i\}_i$ converges to a lattice $\Delta$ iff
$\Delta_i$ has a basis $\{b_1^{(i)},\ldots, b_{m+n}^{(i)}\}$ and $\Delta$ has a basis $\{b_1,\ldots, b_{m+n}\}$ such that
\begin{equation}\label{basis}
 \lim_i b_1^{(i)}=b_1,\ldots,\lim_i b_{m+n}^{(i)}=b_{m+n}.
\end{equation}

 $ M_{m,n}= \mathbb R^{mn}$ stands for the space of        $m\times n$ 
matrices with real 
entries. Recall that $\|\cdot \|$ stands for the sup norm of $\mathbb R^{k}$
and $B_s(x)$ (or $B_s$ if $x=0$) stands for the ball of radius
$s$ centered at $x$ under the sup norm. 
 There is a map
\begin{equation}\label{linkho}
 \phi:M_{m,n}\to SL(m+n,\mathbb R)
\end{equation}
which sends $Y\in M_{m,n}$ to the block matrix $\left(
\begin{array}{cc}
 I_n & 0 \\
Y & I_m
\end{array}
\right)$. 
Let $N$ be a positive integer and $t=\log N$. We set 
\[
 a_N=\left(
\begin{array}{cc}
 e^{-tm}I_n & 0 \\
0 &  e^{tn}I_m
\end{array}
\right)=\left(
\begin{array}{cc}
 N^{-m}I_n & 0 \\
0 &  N^nI_m
\end{array}
\right).
\]

Recall $Y\in \mathrm{DI}_\sigma$  iff there exist
$\mathbf{q}\in\mathbb Z^m \backslash \{0\}$ and $\mathbf{p}
\in \mathbb Z^n$ such that
\begin{equation}\label{diric1}
 {N^m}\|\mathbf q Y+\mathbf p\|\le {\sigma} \quad\mbox{and} \quad 
N^{-n}\|\mathbf q\|\le \sigma   
\end{equation}
for $N$ large enough.
(\ref{diric1}) is equivalent to 
\begin{equation}\label{diric2}
 \|(\mathbf p,\mathbf q)\phi(Y)a_N^{-1}\|=\|
\left(
N^m(\mathbf p+\mathbf q Y),N^{-n}\mathbf q
\right)\|\le \sigma.
\end{equation}
So $Y\in\mathrm{DI}_\sigma$
 iff 
\begin{equation}\label{app1}
 \min\{\|\xi\|:\xi\in\mathbb Z^{m+n}\phi(Y)a_N^{-1},\ \xi\neq { 0}\}
\le \sigma
\end{equation}
for all large enough $N$ depending on $Y$ and $\sigma$.

Let 
\[
K_\sigma=\{\Delta\in \Omega:\min_{\xi\in \Delta\backslash \mathbf{0}}
\|\xi\|> \sigma\}
\quad \mbox{and}\quad \mathbf x=
\mathbb Z^{m+n}=\Gamma\in X.
\]
It is well-known that if $0<\sigma<1$, then $K_\sigma$ is an open neighborhood of 
$\mathbf {x}$. Also, the larger the $\sigma$ is, the smaller the set $K_\sigma$ would be.
 In these notations (\ref{app1}) 
is the same as \begin{equation}\label{app2}
               \mathbf x \phi(Y)a_N^{-1} \not\in K_\sigma.
               \end{equation}
So DT can  be $\sigma$-improved for $Y$ if (\ref{app2})
holds for all  $a_N$ with large enough $N$. 
Let $\tau:M_{m,n}\to X$ be the map which sends $Y$ to $\mathbf{x}
\phi(Y)$. For $a\in G$, we use $T_a:X\to X$ to denote the map that sends $x\in X$
to $xa^{-1}$.

\begin{theorem}\label{diophantine}
 Let $\mu$ be a locally finite measure on $M_{m,n}$ and 
$T=T_{a}$ where $a=a_M$ for some integer $M>1$. If
there exists $s_0>0$ such that for any $s<s_0$ and any ball $B_s(x)$ 
one has
\begin{equation}\label{econdit}
 \lim_{k\to \infty}\frac{1}{k}\sum_{l=0}^{k-1}T_\ast^l(\tau_* 
(\mu|_{B_s(x)}))=\mu(B_s(x))m_X,
\end{equation}
then DT can not be improved for $\mu$ almost 
every element.
\end{theorem}
\begin{proof}
 We need to show for any $0<\sigma<1$, DT can not be $\sigma$-improved for $\mu$ 
almost 
every $Y$. So let us fix some $0<\sigma<1$ and prove $\mu(\mathrm{DI}_\sigma)=0$. 
Since $\mu$ is locally finite, there are sufficiently large real numbers $R$ such that $\mu(\partial B_R)=0$. So it suffices to prove $\mu(\mathrm{DI}_\sigma
\cap B_R)=0$ if $\mu(\partial B_R)=0$. Let us  fix such a positive number $R$. 

\noindent
\textbf{Claim}: there exists $0<\tau<1$ depending on $\sigma$ such that for any $s<s_0$ one has 
\begin{equation}\label{62.1}
\mu(\mathrm{DI}_\sigma\cap B_s(x))\le \tau  \mu(B_s(x))
\end{equation}
for any $x\in M_{m,n}$. Let us assume the claim for the moment and prove
$\mu(\mathrm{DI}_\sigma
\cap B_R)=0$. Suppose otherwise, then  we may choose an open subset $U$ of $B_R$ containing $ 
\mathrm{DI}_\sigma\cap B_R$ such that
\begin{equation}\label{62.2}
 \mu(U)<\frac{1}{\tau}\mu(\mathrm{DI}_\sigma\cap B_R).
\end{equation}
Since $\mu$ is locally finite, $U$ can be covered (measure theoretically)
by countably many disjoint balls  $B_{s_i}(x_i)\subset U$ for $s_i<s_0$
and $x_i\in X$.
By (\ref{62.1}),
\begin{equation}\label{62.3}
 \mu(U)=\sum_i\mu(B_{s_i}(x_i))\ge \frac{1}{\tau}\sum_i
\mu(\mathrm{DI}_\sigma\cap B_{s_i}(x_i))=\frac{1}{\tau}\mu(\mathrm{DI}_\sigma\cap B_R).
\end{equation}
By  (\ref{62.2}) and (\ref{62.3}), we have
\[
 \mu(U)>\mu(U).
\]
This contradiction shows that $\mu(\mathrm{DI}_\sigma\cap B_R)=0$.

Let us prove the claim. We fix $0<s<s_0$ and some $x\in M_{m,n}$.
Since $\sigma<1$,  $K_\sigma$ is an open
neighborhood of $\mathbf x$. Hence there exists $\epsilon>0$ such that
\[
 m_X(K_\sigma)>\epsilon.
\]
So there exists a continuous function $0\le f\le 1$ such that 
\[
 \mathrm{supp}(f)\subset K_\epsilon \quad \mbox{and}\quad 
\int_X f\,\mathrm{d}m_X >\frac{\epsilon}{2}.
\]
We apply this $f$ for (\ref{econdit}), then 
\begin{equation}\label{app4}
 \lim_{k\to \infty}\frac{1}{k}\sum_{l=0}^{k-1} 
\int_{B_s(x)} f(\mathbf x\phi(b)a^{-l}_M)\, \mathrm d \mu(b)=\mu(B_s(x))             \int_X f\, \mathrm d m_X  >\frac{\epsilon}{2}\mu(B_s(x)).
\end{equation}
Let \[
     E=\{b\in B_s(x): \mathbf x\phi(b)a^{-l}_M\not \in K_\sigma  \quad
\mbox{for } l \mbox{ large enough}\}.
    \]
As in (\ref{app2}),
if $b\in \mathrm{DI}_\sigma$, then $\mathbf x\phi(b)a_N^{-1}\not\in K_\sigma$
for all large $N$. In particular 
\[
\mathbf x\phi(b)a^{-l}_M=\mathbf x\phi(b)a^{-1}_{M^l}\not\in K_\sigma
\]
for $l$ large enough. So we have
\begin{equation}\label{app5}
\mathrm{DI}_\sigma\cap B_s(x)\subset E
\end{equation}
From the definition of $E$ and $f$ we see that 
\begin{equation}\label{app6}
 \lim_{k\to \infty}\frac{1}{k}\sum_{l=0}^{k-1}f(\mathbf x\phi(b)a^{-l}_M)
=0
\end{equation}
if $b\in E$.
Note as a  function  of $b$, $\frac{1}{k}\sum_{l=0}^{k-1}f(\mathbf x\phi(b)a^{-l}_M)$
is bounded above by the constant function $1$, so 
the dominated convergence theorem implies
\begin{equation}\label{app2.145}
 \lim_{k\to \infty}\frac{1}{k}\sum_{l=0}^{k-1} 
\int_{E} f(\mathbf x\phi(b)a^{-l}_M)\, \mathrm d \mu(b)=0.
\end{equation}
By (\ref{app4}) and (\ref{app2.145}),
\[
 \mu(B_s(x)\backslash E)\ge \lim_{k\to \infty}\frac{1}{k}\sum_{l=0}^{k-1} 
\int_{B_s(x)\backslash E} f(\mathbf x\phi(b)a^{-l}_M)\, \mathrm d \mu(b)>\frac{\epsilon}{2}
\mu(B_s(x)).
\]
Combine this with (\ref{app5}), we have 
\begin{equation}\label{6.2.8}
 \mu(\mathrm{DI}_\sigma\cap B_s(x))\le 
\mu(E)\le 
\left(1-\frac{\epsilon}{2}\right)
\mu(B_s(x)).
\end{equation}
Since $\epsilon$ only depends on $\sigma$, we may set $\tau=1-\frac{\epsilon}{2}$. This completes the proof of the claim.
\end{proof}
\begin{remark}\label{rdiophantine}
  Suppose $\mathrm{supp}(\mu)$ is contained in a compact set 
$A=\mathrm{clo}(B_R(y))$  for some $R>0$ and $y\in M_{m,n}$
such that  $\mu(\partial A)=0$. It is easy to see from  the proof of Theorem \ref{diophantine} that it suffices to assume (\ref{econdit}) holds for $B_s(x)\subset A$. 
\end{remark}

\begin{theorem}\label{conclut}
Let $\mu$ be a Borel probability measure on $[0,1]^{mn}\subset M_{m,n}$ with
local maximal dimension. If $\tau_*\mu$ has no  loss of mass on average with
respect to $T=T_{a}$ where $a=a_M$ for some integer $M>0$, then DT can not be 
improved for $\mu$ almost every element.
\end{theorem}
\begin{proof}
 By Proposition \ref{decompo}, there exists $s_0>0$ such that
if $B_s(x)\subset J$ and $\mu(B_s(x))\neq 0$ for some $0<s<s_0$, then $\tau_*\nu$ where 
$\nu=\frac{1}{\mu(B_s(x))}\mu|_{B_s(x)}$ has local maximal dimension in the unstable horospherical direction.

Since $\tau_*\mu$ has no loss of mass on average, Lemma \ref{noesma} implies 
$\tau_*\nu$ has no loss of mass on average. Therefore Theorem \ref{wequi}
implies 
\[
 \lim_{k\to \infty}\frac{1}{k}\sum_{l=0}^{k-1}T_\ast^l(\tau_*\nu) =m_X.
\]
That is 
\[
\lim_{k\to \infty}\frac{1}{k}\sum_{l=0}^{k-1}T_\ast^l(\tau_* 
(\mu|_{B_s(x)}))=\mu(B_s(x))m_X.
\]
It is easy to see from the local maximal dimension property of $\mu$ that $\mu(\partial J)=0$,
so the assumptions of Theorem \ref{diophantine} and  Remark \ref{rdiophantine} are satisfied. Therefore
the conclusion follows.
\end{proof}
\noindent
In author's opinion, the assumption of non-escape of mass is superfluous in 
Theorem \ref{conclut}. The following are some facts about non-escape of mass property of a measure with local maximal dimension:
\begin{itemize}
 \item  $m=n=1$ and $G=SL(2,\mathbb R)$. 

\noindent This is proved 
in an unpublished paper of  Einsiedler, Lindenstrauss, Michel and Venkatesh
using hyperbolic geometry of the upper half plane.

\item $m=1,  n=2$ and $G=SL(3,\mathbb R)$. 

\noindent
This is proved by Einsiedler and Kadyrov. In fact they are  working on the non-escape of mass problem uner weaker
assumptions and trying to generalize their method to the cases with $m=1$ or $n=1$.

\item $\mu$ is in addition Federer. 

\noindent
Since Federer and local maximal dimension imply friendly (Theorem \ref{lft}),
Corollary \ref{cfrie} gives the conclusion. This proves Theorem \ref{obtain}.
\end{itemize}

\end{document}